\documentclass[preprint,12pt]{article}

\usepackage{amsmath,amsfonts,amsthm,amssymb,graphicx,color}
 \usepackage{nonfloat}
 \usepackage{amsmath}
 \usepackage{amsfonts}
 \usepackage{amssymb}
 \usepackage{graphicx}
\usepackage{url}

\newtheorem{theorem}{Theorem}[section]

\newtheorem{lemma}{Lemma}[section]

\newtheorem{remark}{Remark}[section]

\newtheorem{example}{Example}[section]
\newcommand{\R}{\mathbb{R}} 
\def\ue{u^\epsilon}
\def\ve{v^\epsilon}

\def\e{\epsilon}
\def\div{\rm div}
\def\mes#1{|#1|}

\def\ratSV{R(P)}

\begin{document}

\title{Mathematical Homogenization in the  Modelling
of  Digestion in the Small Intestine }

\author{M. Taghipoor$^{*,\dagger}$, G. Barles\footnote{
Laboratoire de Math\'ematiques et Physique Th\'eorique (UMR CNRS 6083).
F\'ed\'eration Denis Poisson (FR CNRS 2964) Universit\'e de Tours.
Facult\'e des Sciences et Techniques, Parc de Grandmont, 37200
Tours, France}, Ch. Georgelin$^{*}$,\\ J.-R. Licois$^{*}$
\& Ph. Lescoat\footnote{INRA, UR83 Recherches Avicoles, 37380 Nouzilly, France.}}

\date{}

\maketitle

\begin{abstract}
Digestion in the small intestine is the result of complex mechanical and biological phenomena which can be modelled at different scales. In a previous article, we introduced a system of ordinary differential equations for describing the transport and degradation-absorption processes during the digestion. The present article sustains this simplified model by showing that it can be seen as a macroscopic version of more realistic models including biological phenomena at lower scales. In other words, our simplified model can be considered as a limit of more realistic ones by averaging-homogenization methods on biological processes representation.
\end{abstract}

\vspace{.5cm} {\small \noindent{\bf Key-words :} Digestion in the small intestine, peristalsis, intestinal villi, homogenization, viscosity solutions

\vspace{0.3cm}

\section{Introduction}\label{introduction}
When building a model for digestion in the small intestine, difficulties occur. The first one is the extreme complexity of the mechanical/biological phenomena. Transport of the bolus through the peristaltic waves, feedstuffs degradation by numerous enzymatic reactions and the active/passive absorption of the nutrients by the intestinal wall are known to be the key steps but they are not biologically nor fully understood and neither quantitavely parameterized. Modelling approaches are a way to integrate complex mechanisms representation of these phenomena helping to improve our understanding of them. Since it is almost impossible to build direct experiments for studying the digestion in the small intestine, modelling is a way to test in silico hypotheses that could be challenged through limited in vivo experiments.

 A second difficulty relies on the complex environment within the digestive tract. For example, the intestinal wall plays a key role in the transfer of the digested food in the blood and interferes in the degradation of the bolus via the brush-border enzymes and causes the transit of the bolus by transmitting the pulses coming from the peristaltic waves.

Thirdly, digestion in the small intestine has contrasted but relevant macroscopic and microscopic scales, both in space and time. To give few figures, the length of the small intestine in a growing pig reaches 18 meters, which is a large figure compared to its radius (2-3 centimeters) and even more compared to the size of the villi (around 1 millimeter). In the same way, the bolus stays in the small intestine for several hours, while the efficient peristaltic waves which ensure the transport of the bolus, start approximatively every 12 seconds from the pylorus.

Because of these different scales, a model based on partial differential equations and capturing all the interesting phenomena, would be too complicated and impossible to solve numerically. Therefore we have adopted in \cite{TLGLB} a model based on ordinary differential equations (ode in short) : each bolus of feedstuffs coming from the stomach is identified as a cylinder and the odes describe the evolution of the position and composition of the cylinder. Since digestion could be described by a transport equation (or a system of such equations) with reacting terms, our strategy was essentially to use the Characteristics of this equation. At least numerically this type of Lagrangian method appears to be more efficient. We refer to \cite{TLGLB} for details on our different models since several stages of the modelling process were developped in this paper. 

The aim of the present article is to provide mathematical justifications of some assumptions of the modelling presented in \cite{TLGLB}. We focus on the bolus transport and 
 on the effects related to absorption and enzymatic breakdown by the brush border enzymes, phenomena which are related to averaging/homogenization type processes.

More precisely, in Section~\ref{ch2}, we examine the effects of the pulses generated by the peristaltic waves. Considering that the time scale for these pulses is small compared to the duration of the digestion i.e. that their frequency is high, we rigorously establish that their effect is the same as the one of a constant driving force. This result is biologically very interesting since it allows to get rid of this very small time scale and to do the numerical computations in a much more efficient way opening ways to alternative experimental approaches on digestive tract studies. Related and more general results on the homogenization of odes can be found in L.~C.~Piccinini\cite{Piccinini1978} but we point out that our case does not fall into the scope of \cite{Piccinini1978}.

In Section~\ref{ch3}, we consider the complex phenomena related to the villi and micro-villi : the active/passive absorption by the intestinal wall and the brush border enzymatic reactions. In order to study these phenomena, we introduce a $3$-d model where we focus on the boundary effects. As a consequence, the other phenomena are highly simplified. The lumen of the small intestine is modelled as a cylindrical type, periodic domain whose axis is $\R e_1$, where $e_1 := (1,0,0)$. In order to model the villi, this domain has an highly oscillatory boundary of order $\varepsilon^{-1}$ while its radius is of order $\varepsilon$. In this domain, we have a system of parabolic, transport-diffusion equations with oscillatory coefficients for the absorbable and non-absorbable nutrients. The key feature is the Neumann boundary condition which describes the phenomena on the intestinal wall : the effects of the brush-border enzymes together with the active-passive absorption.

Using homogenization method, we prove that, when $\varepsilon$ tends to $0$, this problem converges to a $1$-d system of transport-reaction equation. The key issue is to show how the effects of the diffusion and the degradation- absorption on the highly oscillatory boundary are combined in order to produce the final reaction terms. For the readers convenience, we provide both a formal and a rigorous proof of this result. The formal proof gives rather explicit formulas which can easily be interpreted from the biological point of view. Moreover we point out that, even if we are using a very simplified framework, we show that it captures the key features of the absorption process.

The homogenization methods used in Section~\ref{ch3} are based on viscosity solutions' theory and in particular the ``perturbed test function method'' of L. C.~Evans \cite{Evans1989,Evans1992} : we refer to \cite{BDLS} and references therein for the applications of such methods for problems with Neumann boundary conditions and oscillatory boundary. To our knowledge, it is the first time that such methods are used to obtain a convergence of a $3$-d problem to a $1$-d problem.\\

\noindent{{\bf Acknowledgement.}
The multidisciplinary collaboration on this research project between the INRA Center of Nouzilly and the Laboratoire de Math\'ematiques et Physique Th\'eorique was initiated within and supported by the CaSciModOT program (CAlcul SCIentifique et MOD\'elisation des universit\'es d'Orl\'eans~et de Tours) which is now a Cluster of the french Region Centre.
This collaboration also takes place in a CNRS-INRA PEPS program ``Compr\'ehension et Mod\'elisation du devenir de l'aliment dans le tube digestif``. This work is part of the PhD thesis of Masoomeh Taghipoor, financed by CNRS and INRA.}

\section{Transport Equation}\label{ch2}

Peristalsis is the phenomenon in which a progressive wave of contraction or expansion (or both) propagates along a tube \cite{Mernone2002,Miftahof2007, Randall1997,TLGLB}.
The peristaltic waves are responsible for the fluid dynamics of the contents of the small intestine and can be divided into segmentation and propulsive contractions. The segmentation motion are responsible for mixing the bolus. %\cited{Guyton},\cited{Ganong}
Propulsive contractions are responsible for transporting the bolus through the small intestine.
The effective peristaltic waves generated in the pylorus reach the bolus approximatively every 12 seconds. This is very small compared to the time scale of digestion phenomena which lasts several hours. This causes the observation of very rapid variations in the velocity of the bolus.

In \cite{TLGLB}, the authors present a first simplified model of bolus transport along the small intestine. We use Homogenization Theory to simplify this equation to replace the periodically oscillating velocity by an averaged one (\ref{hvelocity}). This section provides a rigorous mathematical justification of this transport equation.

 \begin{center}
  \includegraphics[width=\textwidth]{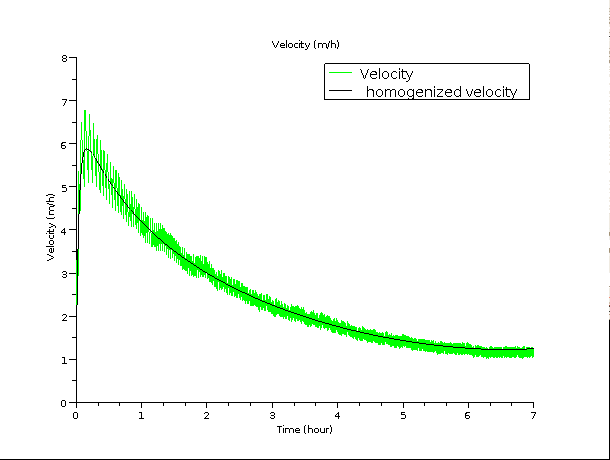}
\figcaption{Periodically oscillating velocity and averaged one. }%
 \label{hvelocity}%-
 \end{center}

\vspace{1cm}
\subsection{Position of the problem}

In this section, we formulate a simplified version of the transport problem. The small intestine is represented by the interval $[0,+\infty)$ and the position of the bolus at time $t$ is given by $x(t) \in [0,+\infty)$. Roughly speaking, $x(t)$ is the distance between the center of bolus and the pylorus.

The bolus is composed of different types of nutrients, say nutrients \\$1, 2, \cdots, K$ and the quantity of nutrient $i$ at time $t$ is denoted by $y_i (t)$ for $i=1, 2, \cdots, K$ and we set $y(t):=(y_1(t), y_2(t),\cdots,y_K(t))$. The variation of the different $y_i$ depends on its production and degradation rate which is summarized through the equation
\begin{equation}\label{equ33}
 \dot{y} (t) =d(x (t) ,y(t))\; ,
\end{equation}
where $d:=(d_1,d_2,...d_n)$ with $d_i : [0,+\infty) \times \R^K \to [0,+\infty)$ a Lipschitz continuous function. Since we are mainly interested in the transport equation in this section, this simple equation is written to fix ideas but also because the transport equation will depend on the composition of the bolus $y$.

The peristaltic waves are created at the pylorus and they travel along the intestinal wall at a quasi-constant velocity : the average wave velocity of each peristaltic wave is about  $c\simeq 7.2 \small m/h$. These waves are periodic of period denoted by $\epsilon \ll 1$ and to model them, we can say that at time $t$, an electric signal of size $\psi(t/\epsilon)$ starts from the pylorus and reaches a point $x$ of the small intestine at time $t + x/c$. Here we assume that $\psi(s) \equiv 0$ if $s \leq 0$ and on $[0,+\infty)$, $\psi$ is the restriction of a smooth, $1$-periodic function on $\R$. 

At time $t$, the bolus is at the position $x(t)$ and is reached by the wave generated at time $s=t-x(t)/c$ whose intensity is $\psi(s/\epsilon)$. we assume moreover that the impact of this pulse on acceleration of the bolus is given by a smooth, positive function $g_\epsilon (s,v, x,y) $ where, as above, $s$ is the time when the pulse was generated, $v$ is the relative velocity of a pulse with respect to the bolus velocity ($v=(c-\dot{x}(s) )/c$ ), $x$ is the position of the bolus and $y$ its composition. 

Indeed, according to \cite{Zhao1997} and \cite{Rivest2000} the efficiency of the peristaltic waves increases with the size of the bolus which is roughly speaking the sum of the $y_i$ for $1\leq i\leq K$, and decreases with the distance from pylorus $x(t)$.

The function $g_\epsilon$ is also $\epsilon$-periodic in $s$, we emphasize this fact by writing
$$g_\epsilon (s,v,x,y)= g (s/\epsilon,v, x,y) \; ,$$
where $g(s,v,x,y)$ is a smooth function which is $1$-periodic in $s$ for $s>0$.

Taking into the friction inside the small intestine as in \cite{TLGLB} through a $-k(t)\dot{x}(t)$-term where $k(t)>0$ for any $t$, the equation for the transport of the bolus reads
\begin{equation}\label{equ29}
\ddot{x}(t)=g\left(\epsilon^{-1}(t-x(t)/c),1-\dot{x}(t)/c, x(t),y(t)\right)-k(t)\dot{x}(t)
\end{equation}
with $x(0)=0$ and $\dot{x}(0)=v_0$ where $v_0<c$.

Having in mind the example of a water wave in a channel, if the bolus velocity is the same or is close to the wave one, then obviously the peristaltic wave will have either no effect or at least a small effect on the bolus velocity. Translated in term of $g$, this means that $g(t,0, x,y)=0$ and even $g(t,v, x,y)=0$ if $v\leq 0$. Thus there exist a smooth function $\tilde{g} : \R \times \R \times  [0,+\infty) \times \R^K \to [0,+\infty)$
such that $g(s,v, x,y)=\tilde{g}(s,v, x,y)v$. We notice that, since we assume $g$ to be positive, then $\tilde{g}(s,v, x,y)\geq 0$ if $v\geq 0$ while we have $\tilde{g}(s,v, x,y)\equiv 0$ if $v\leq 0$ .
 
Because of the dependence of Equation~(\ref{equ29}) on $\epsilon$, we denote their solutions by $x^\epsilon, y^\epsilon$ and our aim is to study the behavior of these solutions as $\epsilon$ tends to $0$, and in particular the behavior of $x^\epsilon$.
\subsection{The Asymptotic Behavior}

We first rewrite the equation satisfied by  $x^\epsilon, y^\epsilon$. For Equation~(\ref{equ29}), we have
\begin{equation}\label{equ7}
\ddot{x}^{\epsilon}=
(1-\dot{x}^{\epsilon}/c)\tilde{g}(\epsilon^{-1} (t-x^{\epsilon}/c) ,1-\dot{x}^{\epsilon}/c, x^{\epsilon},y^{\epsilon})
-k(t)\dot{x}^{\epsilon}
\end{equation}
while Equation (\ref{equ33}) reads
 \begin{equation}\label{equ35}
 \dot{y}^{\epsilon}=
d(x^{\epsilon},y^{\epsilon}).
\end{equation}
The initial conditions are 
\begin{equation}\label{equ34}
 x^{\epsilon}(0)=0, 
\quad\dot{x}^{\epsilon}(0)=v_0,
 \quad  y^\epsilon(0)=y_0,
\end{equation}
where $v_0 <c$ because of physiological reasons.

In order to formulate our result, we introduce the function $F(t,V,X,Y)$ given by
\begin{equation}\label{equ31}
F(t,V,X,Y)=\int_0^t\tilde{g}(s,V,X,Y)ds
\end{equation}
and, recalling that $\tilde{g}(s,V,X,Y)$ is $1$-periodic for $s\geq0$, we denote the averaged of $F$ over a period by $\bar{F}(V,X,Y)$. Of course we have
$$\bar{F}(V,X,Y)=F(1,V,X,Y)=\int_0^1\tilde{g}(s,V,X,Y)ds,$$
and $F, \bar{F}$ are smooth functions since $\tilde g$ is a smooth function.
\begin{example}\label{ex1}
 In \cite{TLGLB}, the authors introduce the the function $g$ as 
\begin{equation*}
 g(s,v,x,y):= \dot{\psi}(t-x/c)\cdot v\frac{c_0+c_1y}{a+bx},
\end{equation*}
for the real non-negative values $c_0, c_1, a$ and $b$. Where $\displaystyle  y= \Sigma_{i=1}^{n} y_i$.
\end{example}

Our main result is the
\begin{theorem}\label{th1}
 Let $(x^\epsilon, y^\epsilon)$ the unique solution of equations (\ref{equ7})-(\ref{equ35})-(\ref{equ34}), then the sequence $(x^\epsilon, y^\epsilon )_{\epsilon >0}$ converges strongly in $C^1([0,T],[0,+\infty))$ to $(x,y)$ the unique solution of the averaged system of equations
\begin{align}\label{equ30}
\ddot{x}(t)=
&\dfrac{c-\dot{x}(t)}{c}\bar{F }( 1-\dot{x}(t)/c,{x}(t),{y}(t))
-k(t)\dot{x}(t)\\
\dot{y}(t)=&d(x(t),y(t)) \nonumber
\end{align}
with the initial conditions
\begin{equation}
x(0)=0,
\quad\dot{x}(0)=v_0,
 \quad  y(0)=y_0.
\end{equation}
\end{theorem}

The key interpretation of this result is the following : the effect of frequent pulses on the transport of the bolus is the same as the one obtained through an averaged constant signal.

\begin{proof}[Proof of Theorem \ref{th1}] We prove it in two steps : first we obtain various estimates showing that the sequences $(x^\epsilon, y^\epsilon )_{\epsilon >0}$ converge strongly in $C^1$ (at least along subsequences) and, then, in the second step, we prove that they converge to the unique solution of the averaged system (\ref{equ30}) (which will imply that the whole sequence converges by a standard compactness argument).

The following lemma provides a proof of convergence of $x^\epsilon$ and $y^\epsilon$.
\begin{lemma}\label{lem1}
Let $(x^\epsilon, y^\epsilon )_{\epsilon >0}$ the unique solution of (\ref{equ7})-(\ref{equ35})-(\ref{equ34}). Then $x^\epsilon, y^\epsilon$ are uniformly bounded in $C^2$ and therefore there exists a subsequence which is converging strongly in $C^1$ and such that $\ddot{x^\epsilon}$ is converging in $L^{\infty}$ weak-$*$.
\end{lemma}

\begin{proof}[Proof of Lemma]
%First we point out that by equation (\ref{equ7}), if $\dot{x}^\epsilon$ is bounded then, since  $x^\epsilon(t) $ is a continuous increasing function of $ t\in [0,T]$, the boundedness of $\dot{x}^\epsilon$ gives the boundedness of $x^\epsilon$.

We first prove that $\dot{x^\epsilon}(t)\leq c$. To this aim, we define the positive function $\phi (t)$ as follows   
\begin{equation}\label{equ9}
\phi(t)=
(\dot{x^\epsilon }(t)-c)^+ =
\left\{
 \begin{aligned}
 & \dot{x^\epsilon }(t)-c \quad &\text{if} \quad\dot{x^\epsilon }(t)-c>0\\
 & 0 \quad &\text{otherwise},
\end{aligned}
\right.
\end{equation} 
then multiply the both sides of equation (\ref{equ7}) by $\phi(t)$% and recall the positivity of $v$
$$\ddot{x}^{\epsilon}(\dot{x^\epsilon }(t)-c)^+ =\dfrac{c-\dot{x}^{\epsilon}}{c}\tilde{g}\left(\dfrac{tc-x^{\epsilon}}{c \epsilon},1-\dot{x}^{\epsilon}/c, x^{\epsilon},y^{\epsilon}\right)(\dot{x^\epsilon }(t)-c)^+ -k(t)\dot{x}^{\epsilon}\left(\dot{x^\epsilon }(t)-c\right)^+\; . $$
The right-hand side of this equation is negative since $\tilde{g}(s,V,X,Y)$ is non negative if $Y \geq 0$, and $k(t)>0$, therefore
$$ \ddot{x^\epsilon}(t)(\dot{x^\epsilon }(t)-c)^+ \leq 0$$
which is equivalent to
$$\dfrac{1}{2}\dfrac{d}{dt}(\phi^2(t))\leq 0 \; .$$
The function $\phi^2(t) $ is therefore a decreasing function. Furthermore, since $v_0 <c$, we have 
 $$\phi^2 (t) \leq \phi^2(0)= [(v_0-c)^+]^2 = 0\; ,$$
which yields the result.

Using the same method with $(\dot{x})^- = \max(-\dot{x},0)$, we can prove that $\dot{x}$ is a non-negative function. 

Gathering these information, we obtain that, for any $t$
$$0 \leq x^\epsilon (t)\leq ct \; , \; 0 \leq \dot x^\epsilon (t)\leq c \; ,$$
and therefore the sequence $(x^\epsilon)_{\epsilon >0}$ is uniformly bounded and equicontinuous on $[0,T]$. Using these informations and the equations for the $y^\epsilon$, we also see that the $y^\epsilon$ are also uniformly bounded in $C^1$ (and even in $C^2$) and coming back to the $x^\epsilon$ equation we see also that the $x^\epsilon$ are also uniformly bounded in $C^2$.

Consequently the Arzela-Ascoli compactness criterion ensures that there exists a subsequence $(x^{\epsilon_j }, y^{\epsilon_j })$ which converges in $C^1$. Moreover, since $\ddot{x}^{\epsilon}$ is bounded in $L^\infty$, we can also extract a subsequence such that $\ddot{x}^{\epsilon_j }$ converges in the $L^\infty$ weak-$*$ topology.
%$ \to  \bar{x} \qquad \text{uniformly in } [0,T]$$
%for some subsequence $\{x^{\epsilon_j }\}_j $ and some limit function
%$\bar {x} \in C([0,T])$.
\end{proof}
% \begin{corollary}
% The function $\dot{x}^\epsilon(t)$ tends to  $\dot{\bar{x}}(t)$ for $\epsilon\ll1$.
% \end{corollary}
% 
We return now to the proof of Theorem~\ref{th1}. To simplify the exposure, we still denote by $(x^{\epsilon}, y^{\epsilon })$ the converging subsequence $(x^{\epsilon_j }, x^{\epsilon_j })$ and we denote by $(x,y)$ the limit.
By inserting the Definition (\ref{equ31}) into Equation (\ref{equ7}), we get
\begin{equation*}
 \ddot{x}^\epsilon(t)=
(1-
\dot{x}^\epsilon/c)\dfrac{\partial F}{\partial t}( \epsilon ^{-1}(t-x^{\epsilon}/c),1-
\dot{x}^\epsilon/c,x^{\epsilon},y^{\epsilon})
-k(t)\dot{x}^\epsilon(t) 
\end{equation*}
therefore, using the notation $v^\epsilon= 
1-\dot{x}^\epsilon/c$ and dropping most of the variables to simplify the expressions, we have
\begin{equation*}
\ddot{x}^\epsilon(t)=
 \epsilon\dfrac{d}{dt}\left[{F}(\epsilon ^{-1}(s-x^{\epsilon}/c),v^\epsilon,x^{\epsilon},y^{\epsilon})\right]
-\epsilon\dot{v^\epsilon}\frac{\partial F}{\partial V}
-\epsilon\dot{x}^\epsilon\frac{\partial F}{\partial X}
-\epsilon\dot{y}^\epsilon\frac{\partial F}{\partial Y}-k\dot{x}^\epsilon
\end{equation*}
and then integrate the both sides of equation  over $[0,t]$
\begin{align*}
\int_0 ^t \ddot{x}^{\epsilon}ds=
\epsilon \int_0 ^t \dfrac{d}{ds}(F( \epsilon ^{-1}&(s-x^{\epsilon}/c),v^\epsilon,x^{\epsilon},y^{\epsilon}))ds\\
&-\epsilon\int_0 ^t (\dot{v^\epsilon}\frac{\partial F}{\partial V}
+\dot{x}^\epsilon \frac{\partial F}{\partial X}
+\dot{y}^\epsilon \frac{\partial F}{\partial Y})ds
-\int_0 ^t k\dot{x}^\epsilon ds
\end{align*}
which leads to
\begin{align*}
\dot{x}^{\epsilon}(t) -v_0 =
\epsilon F( \epsilon ^{-1}&(t-x^{\epsilon}/c),1-
\dot{x}^\epsilon/c,x^{\epsilon},y^{\epsilon})\\
&-\epsilon\int_0 ^t (\dot{v^\epsilon}\frac{\partial F}{\partial V}
+\dot{x}^\epsilon \frac{\partial F}{\partial X}
+\dot{y}^\epsilon \frac{\partial F}{\partial Y})ds
-\int_0 ^t k\dot{x}^\epsilon ds
\end{align*}
since $F(0,V,X,Y)=0$ for any $V,X\in \R$ and $Y\in \R^K$.

Now we have to let $\epsilon$ tend to $0$. First, since $\tilde{g}$ is periodic, it is standard to prove that 
$$ \epsilon {F}( \epsilon ^{-1}t,V,X,Y) \to \bar{F}(V,X,Y)t\; ,$$
locally uniformly. Especially, it is easy to see that if $n\leq\epsilon^{-1}t<n+1$, therefore $\displaystyle \epsilon {F}( \epsilon ^{-1}t,V,X,Y) \to \e[n\bar{F}(V,X,Y)+O(1)] .$

In the same way, because of the definition of $F$ and the regularity properties of $\tilde{g}$, for $\xi = V,X,Y$ we also have
$$ \epsilon {\frac{\partial F}{\partial \xi}}( \epsilon ^{-1}t,V,X,Y) \to \frac{\partial \bar{F}}{\partial \xi}(V,X,Y)t\; \quad \hbox{locally uniformly}.$$    
As a consequence, since $x^{\epsilon}$ and $y^{\epsilon}$ are converging respectively  to $x$ and $y$ in $C^1$, we have also
$$
\epsilon F ( \epsilon ^{-1}(s-x^{\epsilon}(s)/c),v^\epsilon (s),x^{\epsilon} (s) ,y^{\epsilon} (s)) \to \bar{F}(v (s) ,x (s),y(s))(s-x(s)/c),
$$
uniformly on $[0,T]$, where $v = 1-\dot{x}/c$. And the same is true, replacing $F$ by $\displaystyle\frac{\partial F}{\partial \xi}$ and $\bar{F}$ by $\displaystyle \frac{\partial \bar{F}}{\partial \xi}$.

From these properties, it is easy to deduce that
$$\epsilon\int_0 ^t (\dot{x}^\epsilon \frac{\partial F}{\partial X}
+\dot{y}^\epsilon \frac{\partial F}{\partial Y})ds \to 
\int_0 ^t (s-x/c)(\dot{x} \frac{\partial \bar{F}}{\partial X}
+\dot{y} \frac{\partial \bar{F}}{\partial Y})ds \;\hbox{as  }\epsilon \to 0\; ,$$
for any $t \in [0,T]$.

On the other hand, $\dot{v^\epsilon} = -\ddot{x}^\epsilon/c$ converges in the $L^\infty$ weak-$*$ topology to $\dot{v}$ and therefore
$$ \epsilon\int_0 ^t \dot{v^\epsilon} \frac{\partial F}{\partial V} ds \to \int_0 ^t (s-x/c)\dot{v} \frac{\partial \bar{F}}{\partial V} ds \; ,$$
for any $t \in [0,T]$.

Gathering all these informations, we finally obtain
\begin{align}
\dot{x}(t)-v_0=&
(t-x(t)/c)\bar{F}(
1-\dot{x}(t)/c,x(t),y(t)) \nonumber \\
&-\int_0 ^t (s-x/c)(\dot{v} \frac{\partial \bar{F}}{\partial V} ds 
+\dot{x} \frac{\partial F}{\partial X}
+\dot{y} \frac{\partial F}{\partial Y})ds \nonumber  \\
&
-\int_0 ^t k(s)\dot{x}(s) ds \label{equ42}
\end{align}

The right-hand side being $C^1$ in $t$, we deduce that $x$ is a $C^2$-function and derivating the both side of (\ref{equ42}), we have the equation 
 \begin{equation*}
\ddot{x}(t)=(1-\dot{x}(t)/c)\bar{F}(
1-\dot{x}(t)/c,x(t),y(t))
- k(t)\dot{x}(t).
\end{equation*}

%  \begin{equation*}
% \ddot{x}(t)=(1-\dot{x}/c)
% \bar{F}( 1-\dot{x}/c,x,y)
% - k(t)\dot{x}(t)
% \end{equation*}
\end{proof}
\section{On the Effects of Intestinal Villi}\label{ch3}

As mentioned in the introduction, in the 1-d model of digestion presented in \cite{TLGLB}, we take into account the absorption effect by a simple absorption term through a Michaelis-Menten type nonlinearity and this can be assumed unreasonable when compared to the complexity of the involved phenomena.  The same can be said for the enzymatic breakdown by the brush-border enzymes. Consequently the first aim of this section consists in giving some rigorous justification of these choices.

Our effort is therefore to find an appropriate system of equations describing the different effects of the structure and the spatial distribution of intestinal villi on these key phenomena of digestion. Therefore we introduce a 3-d toy model which takes into account the complex geometry of the small intestine as well as all these boundary effects, but this implies unavoidable simplifications on the transport process.

We start by a short presentation of the small intestine anatomy followed by introducing the three dimensional toy model of digestion.

A large number of villi and micro-villi are present on the surface of the small intestine. 
 Their role is to enlarge the digestive and absorptive area in the small intestine. They increase the area of the small intestine at least 500 times (\cite{Randall1997}). The absorptive surface of the villi contains the brush border enzymes which are responsible of the final step of degradation (surfacic degradation) for some nutrients. This increase is therefore a key issue in the process of nutrients degradation and absorption.(\cite{Keener2008}). 

These finger like villi are covered by epithelial cells. They consist of absorptive, goblet and entero-endocrine cells. The epithelial cells are produced in crypts, they migrate and become mature from the crypts to the tips of the villi(\cite{YAMAUCHI(2007)}).  More precisely, the absorption rate is also proportional to the distance of each of the villi from its tip.

A microscopic observation of the small intestine surface is necessary in order to give realistic absorption and degradation shapes. The spatial aspect of absorption related to the distribution of villi and their absorption capacity is often neglected in modelling of digestion. In these models, absorption and degradation are modelled by a constant rate or a Michaelis-Menten process (see \cite{TLGLB}, Logan \cite{LOGAN2002}, ...). 

As shown in figure (\ref{smintest.jpg}) the size of the period is small compared to the size of the unfold small intestine which is around $18$ meters. We consider, for the sake of simplicity, that the villi are distributed periodically in the inner surface of the small intestine.

 \begin{center}
  \includegraphics[width=\textwidth]{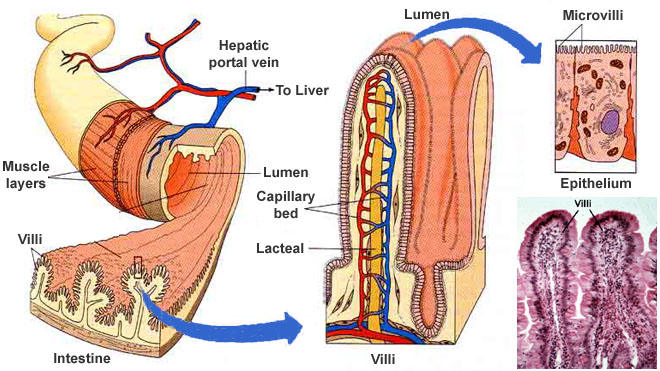}
\figcaption{The different scales on intestinal anatomy relevant to our model. }%
 \label{smintest.jpg}%
 \end{center}

\vspace{1cm}
In this section, we seek a macroscopic description of digestion in the small intestine by taking into account all the effects of the presence of the villi in microscopic scale.  Our approach is based on an asymptotic analysis, as $\epsilon$ goes to zero. The absorption rate of the limit problem is said to be the homogenized absorption rate.

\subsection{Position of problem}\label{3.2}

The small intestine is assumed to be an axisymmetric cylindrical tube with a rapidly varying cross section.
In order to describe it, we first introduce an axisymmetric, smooth domain $\Omega$ which is confined in a cylinder of radius $r >0$. More precisely, we assume
\begin{equation*}
 \{(x_1,x_2,x_3)\in \R^3 \, \mid \, x_2=x_3=0\} \subset  \Omega \subset \{(x_1,x_2,x_3)\in \R^3 \, \mid \, x_2^2+x_3^2=r^2 \}\, .
\end{equation*}
In addition, we assume that $\Omega$ is periodic in the $x_1$-direction (say $1$-periodic), namely
$(x_1 + 1,x_2,x_3)\in \Omega$ if $(x_1,x_2,x_3)\in \Omega$.

The small intestine is represented, for some $0<\epsilon \ll 1$ by the domain $\Omega _\epsilon$ given by
\begin{equation}\label{equ32}
 \Omega _\epsilon = \epsilon \Omega \cap \{x_1 \geq 0\}
\end{equation}
The figure \ref{villi} is a simple representation of this domain :
\begin{center}
    \mbox{\includegraphics[scale=0.72]{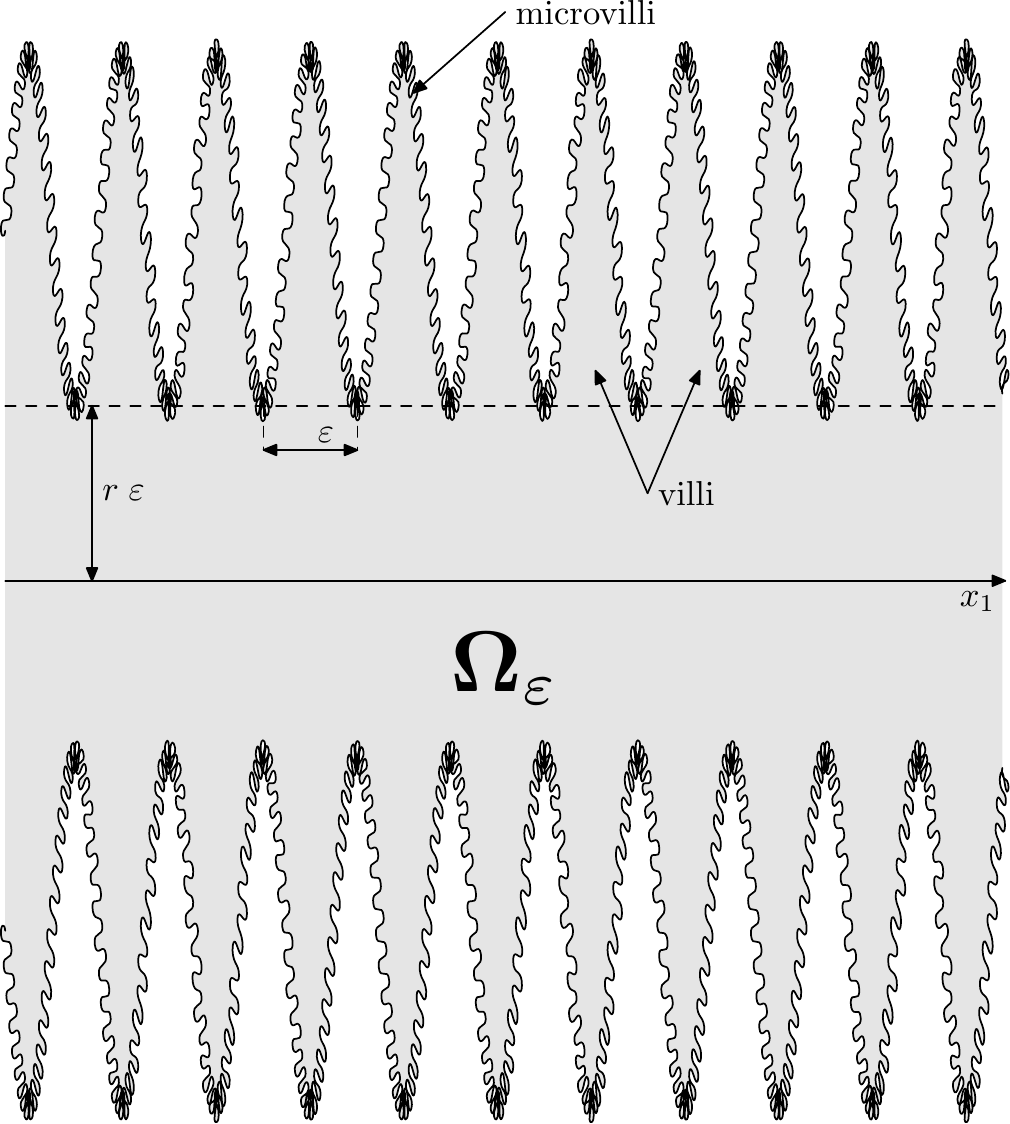}}
\figcaption{A simple example of the domain $\Omega _\e$. The oscillations on the on the villi represents the microvilli.}%
\label{villi}
\end{center}

In this definition, the small intestine has an infinite length. However this assumption is not a real restriction, since we focus on the local absorption-degradation processes. Moreover, the $x_1=0$ part of the boundary corresponds to the pylorus and $\epsilon \partial \Omega$ to the villi. It is worth pointing out that $ \Omega _\epsilon$ is $\epsilon$-periodic in the $x_1$ direction, the parameter $\epsilon$ characterizes the distance between the villi and thus it is natural to assume it to be very small.

A simple example of $\Omega _ \epsilon$ in cylindrical coordinates, can be the following 
$$\Omega_{\epsilon}=\{(r,\theta,z)\, \,\text{s.t} \,\,\mid r\mid \leq \epsilon+\epsilon \psi (z/\epsilon,\theta), \, \, z\geq 0 ,\, \, \theta\in [0,2\pi]\}$$ 
where $z$ plays here the role of $x_1$ and $\psi (z,\theta)$ is a $1$-periodic function of $z$.

We introduce two functions $\ue, \ve : \Omega_{\epsilon} \times [0,T]  \to \R$ for describing the evolution of the concentration of feedstuffs in the small intestine. For $x\in \Omega_\epsilon$ and $t\in [0,T]$, $\ve (x,t)$ denotes the concentration of the large feedstuffs molecules which are transformed into absorbable nutrients after different enzymatic reactions. The quantity $\ue(x,t)$ represents the concentration of produced nutrients at position $x$ at each time $t$.

The evolution of substrates $\ue$ and $\ve$ in the intestinal lumen is due to $(i)$ their diffusion by Fick's law, $(ii)$ their propagation through intestinal lumen  by a given velocity coming from the peristaltic waves and $(iii)$ the enzymatic reactions which transform $\ve$ to $\ue$  both inside the intestinal lumen but also on the intestinal wall by the brush-border enzymes.  When these reactions take place in the intestinal lumen, we call them volumic transformation, while we talk about surfacic transformation when they take place on the villi.

The rate of the volumic reactions depends on the concentration of feedstuffs and also enzymes activity at time $t$ and at $x$, namely $\zeta(x,t)$, where $\zeta : [0,\infty) \times [0,T] \to \R $ is a continuous, positive and bounded function. There is a limitation in the transformation which is described by $\varphi : \R \to \R$, which is a bounded, increasing and Lipschitz continuous function such that $\varphi(s)=0$ if $s\leq 0$. These assumptions on $\zeta$ and $\varphi$ are denoted by (T1) in the sequel.

 Taking into account the three above-mentioned phenomena, the equation for the evolution of concentration of the non-absorbable feedstuffs molecules in the intestinal lumen reads
\begin{equation}\label{equfeed}
    \dfrac{\partial v^\epsilon}{\partial t}=\omega_\epsilon \Delta v^\epsilon-c(x_1,x/\epsilon,t)D v^\epsilon-\zeta(x_1,t)\varphi(v^\epsilon) \quad \text{in} \quad \Omega_\epsilon \times (0,T)
 \end{equation}
 while for the absorbable nutrients, we have
\begin{equation}\label{feedtonut1}
  \dfrac{\partial u^\epsilon}{\partial t}=\chi_\epsilon \Delta u^\epsilon-c(x_1,x/\epsilon,t)D u^\epsilon+  \zeta(x_1,t)\varphi(v^\epsilon)\quad \text{in} \quad \Omega_\epsilon \times (0,T).
\end{equation}

 The first terms of the right-hand-sides of the above equations, where $\Delta$ denotes the usual Laplacian\footnote{If $\phi$ is a smooth function, $\displaystyle \Delta \phi = \frac{\partial^2 \phi}{\partial x_1^2}+\frac{\partial^2 \phi}{\partial x_2^2}+\frac{\partial^2 \phi}{\partial x_3^2}$}, are diffusion terms.
The diffusion coefficients of large molecules of feedstuffs and small molecules of nutrients are denoted by $\omega_\epsilon$ and $\chi_\epsilon$ respectively.

It is shown that, for a fixed temperature, the diffusion coefficient $d$ is inversely proportional to the molecular weight, to be more precise for a spherical molecule we have 
$$d=\frac{kT}{3\mu}(\frac{\rho}{6\pi M})^{1/3}$$
in which $k$ is Boltzmann constant, $T$ is the intestinal temperature, $\mu$ the viscosity of the the intestinal liquid, $\rho $ the molecule density and $M$ the molecular mass. For fixed $T$ and $\mu$, this constant is very small because of the very small value of $kM^{-1/3}$ \cite{Keener2008}.

 For reasons explained in section \ref{rr}, we assume that
$$ \omega_\epsilon :=\epsilon \omega  \quad\hbox{and}\quad \chi_\epsilon:=\epsilon \chi \; ,$$ for some constants $\omega,\chi>0$. Since the nutrients molecules are smaller than feedstuffs particles, we also have $\omega \leq\chi$. 
 The second terms of the right-hand sides are transport terms. The $C^1$-function $c: [0,+\infty) \times \Omega \times [0,T] \to \R^3$ is modelling the velocity of substrates which comes from the peristaltic waves. The effect of the peristaltic waves is known to depend on the position in the small intestine and on time, this justifies the dependence of $c(x_1,X,t)$ on $x_1$ and $t$, while the dependence on $X=(X_1,X_2,X_3)$ takes into account the local effects at a lower scale.

A priori the diffusion of bolus is small compared to its velocity through the small intestine and therefore $\omega_{\e}$ and $\chi_\e$ are expected to be smaller than $c(x_1,X,t)$. 

 We assume the function $c$ to satisfy the following properties :\\
{\bf (C1)} The function $c(x_1,X,t)$ is a Lipschitz continuous function which is $1$-periodic in $X_1$ and, if $e_1=(1,0,0)$, then, for any $x_1 \in [0,+\infty)$, $X \in \Omega$ and $t \in [0,T]$,
$$ c(x_1,X,t)\cdot e_1 \geq 0\quad \hbox{and}\quad  \int_P c(x_1,X,t)dX >0\; ,$$
where $P$ is a period in $\Omega$, say $P:=\{X\in \Omega\ ;\ \ 0\leq X_1 \leq 1\}$.

In addition to the regularity properties of $c$, this assumption means that the effect of the peristaltic waves is to move ahead the bolus in the small intestine.\\
{\bf (C2)} For any $x_1 \in [0,+\infty)$, $X \in \Omega$ and $t \in [0,T]$, $\div_X(c)=0$ where $\div_X$ denotes the divergence operator in the $X$-variable only. \\
This second assumption is justified by the incompressibility of the bolus at the microscopic level.\\
{\bf (C3)} For any $x_1 \in [0,+\infty)$, $X \in \partial \Omega$ and $t \in [0,T]$, $c(x_1,X,t)\cdot N(X) =0$, where $N(X)$ denotes the outward, unit normal to $\partial \Omega$ at $X$.\\

This last assumption means that the velocity vector is always tangent to the boundary. It is worth pointing out that, if $X=x/\epsilon$ then $N(X) = n(x)$, therefore it is true both for $X$ in $\Omega$ and for $x$ in $\Omega_\epsilon$. As a consequence of this property, the nutrients reach the boundary only because of the diffusion effects.

Once they reach the boundary, the large particles of feedstuffs can change of chemical structure because of the presence of brush-border enzymes. As we already mentioned above, this effect is called surfacic degradation of feedstuffs and the result is the production of the smaller absorbable molecules of nutrients $\ue$. We assume moreover that a portion $0\leq \beta < 1$ of these nutrients is absorbed instantaneously while the remaining part ($\alpha:=1-\beta$) diffuses in the small intestine. The surfacic degradation is modelled by the Neumann boundary condition
  \begin{equation}\label{equfeedb} 
\omega \frac{\partial \ve}{\partial n}=-\varrho(x_1,X)\ve\quad \text{on}\,\,\partial \Omega_{\epsilon}\times (0,T)
\end{equation}
 where, $\varrho$ is a continuous, positive and $X_1$-periodic function which represents the rate of surfacic degradation.

 On the boundary of the small intestine, there are two main effects for the nutrients $\ue$. We already describe the first one which is a production of nutrients by the surfacic degradation. The second one is the active and passive absorption of nutrients, namely their transport across the intestinal wall to the blood circulation. An active process requires the expenditure of energy, while a passive process results from the inherent, random movement of molecules \cite{Keener2008}. These different categories of absorption as well as the production of $\ue$ from $\ve$ on the boundary construct the boundary condition of Equation (\ref{equfeed})
\begin{equation}\label{equnutb} 
\chi \dfrac{\partial u^\epsilon}{\partial n} =- \eta_p(x_1,x/\epsilon) u^\epsilon-\eta_a(x_1,x/\epsilon,t) g_a(u^\epsilon)+\frac{\alpha}{\omega} \varrho(x_1,x/\epsilon)\ve.
\end{equation} 

The functions $\eta_p$ and $\eta_a$ denote respectively the passive and active absorption rates. Both of them depend on the global position in the small intestine $x_1$ and the local one $x/\epsilon$, by which we take into account the effect of the special physiology of the villi on the absorption rate which has been described at the beginning of this section. The dependence in time in the active absorption $\eta_a$, describes the presence of energy at time $t$. The function $ g_a$ governs the active absorption and depends on the nutrients categories. Typically, it is assumed to be the Michaelis Menten and therefore, it is a bounded, continuous, increasing function. 

We formulate the key assumptions on the functions $\eta_p, \eta_a$ and $g_a$\\
{\bf (T2)} The functions $\eta_p (x_1,X), \eta_a(x_1,X, t)$ are bounded, continuous, positive, $1$-periodic functions in $X_1$, and the function $g_a$ is a bounded, continuous, increasing function with $g_a(s) = 0$ if $s\leq 0$. Moreover, there exists $\underline{\eta}>0$ such that $\eta_p (x_1,X)\geq \underline{\eta}$ for any $x_1 \in [0,+\infty)$, $X \in \Omega$.\\

Finally, we complement the equations with the initial conditions
\begin{equation}\label{emptySI}
 u^{\epsilon}(x,0)=0, \quad v^\e(x,0)=0  \quad\text{in  } \Omega_\epsilon ,
\end{equation}
which means that the small intestine is empty at time $t=0$ and by a Dirichlet boundary condition at $x_1=0$, modelling the gastric emptying, namely
\begin{align}
  v^{\epsilon}(x,t)=v_0(t)  &\quad \text{for  } x_1 = 0, t \in (0,T)\; \label{Pylfeed}\\
  u^{\epsilon}(x,t)=u_0(t)& \quad \text{for  } x_1 = 0, t \in (0,T)\;\label{Pylnut} ,
\end{align}
where $u_0$ and $v_0$ are bounded continuous functions on $[0,T]$ with $u_0(0)=0$ and  $v_0(0)=0$.

\subsection{Formal asymptotic}

 In order to study the limit as $\epsilon \to 0$ of the system (\ref{equfeed})-(\ref{Pylnut}),
%initial-boundary value problems (\ref{feedtonut1})-(\ref{equnutb})-(\ref{emptySI})-(\ref{Pylnut}) and (\ref{equfeed})-(\ref{equfeedb})-(\ref{emptySI})-(\ref{Pylfeed}),
we first argue formally : we consider the following expansions (called ansatz) for the solutions $u^\epsilon$ and $v^\e$
\begin{align}
    & u^\epsilon(x,t)
    =u(x_1,t)+\epsilon u_1(x_1,\dfrac{x}{\epsilon},t)
    +o({\epsilon}) \label{ansatzu}\\
    & \ve(x,t)
    =v(x_1,t)+\epsilon v_1(x_1,\dfrac{x}{\epsilon},t)
    +o({\epsilon})  \label{ansatzv}
\end{align}
where $u_1(x_1,\dfrac{x}{\epsilon},t)$ and $v_1(x_1,\dfrac{x}{\epsilon},t)$ are 1-periodic functions in second variable.

From now on, in order to simplify the notations,  we systematically denote by $X$ the fast variable $x/\epsilon$. On the other hand, the above system can be decoupled and we can first study the asymptotics of $\ve$, namely only the initial-boundary value problem (\ref{equfeed})-(\ref{equfeedb})-(\ref{emptySI})-(\ref{Pylfeed}) and then use the result for studying the behavior of $\ue$ through (\ref{feedtonut1})-(\ref{equnutb})-(\ref{emptySI})-(\ref{Pylnut}). Since we use the same methods in both cases to obtain the homogenization results, we present the details only for the equation of nutrients $\ue$ while we only give the results for $\ve$.

We first plug these expressions of $v^\e$ and $u^\epsilon$ into (\ref{feedtonut1}), and then examine the higher order terms in $\e$. We find
\begin{equation}
    \dfrac{\partial u}{\partial t}= \chi_{\epsilon}(\dfrac{\partial^2 u}{\partial x_1 ^2} +\dfrac{1}{\epsilon}\Delta_{X} u_1)-c(x_1,X,t)(\dfrac{\partial u}{\partial x_1}e_1+D_X  u_1)+\zeta(x_1,t)\varphi(v)+o(1)
\end{equation}
At this stage, we notice that the relevant choice for observing the effects of villi is indeed $\chi_\epsilon =\epsilon\chi$, for some positive constant $\chi$. With this choice, we obtain
\begin{equation}\label{equ13}
      \dfrac{\partial u}{\partial t}= \chi \Delta_{X} u_1-c(x_1,X,t)(\dfrac{\partial u}{\partial x_1}e_1+D_X u_1)+\zeta(x_1,t)\varphi(v)+o(1).
 \end{equation}
The equation for the first corrector $u_1$ (the ``cell problem'') is an equation in the fast variable $X$, i.e. for the functions $X \mapsto u_1(x_1,X,t)$, $x_1,t$ playing the role of parameters. Setting
    $$p:=\dfrac{\partial u}{\partial x_1}(x_1,t)e_1 \quad, \quad \lambda:=-\dfrac{\partial u}{\partial t}(x_1,t) \quad\hbox{and} \quad \delta :=\zeta(x_1,t)\varphi(v),$$ 
and substituting $p$ and $\lambda$ in (\ref{equ13}), we obtain the equation on $\Omega$ 
 \begin{equation}\label{CP1}
    -\chi \Delta u_1+c(x_1,X,t)[p+D_Xu_1]=\lambda +\delta\quad \text{in}\,\, \Omega .
\end{equation}

We argue in the same way for the boundary condition : plugging (\ref{ansatzu}) and (\ref{ansatzv}) into ($\ref{equnutb})$, we obtain
\begin{equation}
    \chi(\dfrac{\partial u}{\partial x_1}e_1 +D_X u_1).n
    =-\left(\eta_p(x_1,X)u+\eta_a(x_1,X,t)g_a(u)-\frac{\alpha}{\omega}\varrho(x_1,X)v\right)+ o(1)\;  .
\end{equation}
Using the introduced notations and recalling that $N(X)=n(x)$, the above equation gives
\begin{equation}\label{equ38}
  (p+D_Xu_1).N
  =-\frac{1}{\chi}\left(\eta_p(x_1,X)u+\eta_a(x_1,X,t)g_a(u)-\frac{\alpha}{\omega}\varrho(x_1,X)v\right)+ o(1)\; .
\end{equation}
 Introducing the notations $\mu:=u(x_1,t)$ and $\nu=v(x_1,t)$ and
$$
\Theta(x_1,X,t,u,v):=
\eta_p(x_1,X)u+\eta_a(x_1,X,t)g_a(u)-\frac{\alpha}{\omega}\varrho(x_1,X)v\; ,
$$
the complete cell problem reads 
\begin{equation}\label{cellpb}
    \left\{
    \begin{aligned}
    -\chi \Delta u_1+c(x_1,X,t)[p+D_Xu_1]=\lambda + \delta \,\, &\,\, \text{in}\,\, \Omega \\
    (p+D_Xu_1)\cdot N =-\frac{1}{\chi} \Theta(x_1,X,t,\mu,\nu) \,\, & \,\,\text{on} \,\,\partial \Omega 
    \end{aligned}\right.
\end{equation}
We assume that this problem has indeed a smooth solution $u_1$ which is $1$-periodic in $X_1$.
Recalling that $\Omega$ is $1$ periodic in the $X_1$ direction and integrating (\ref{cellpb}) over a period $P$ (remarking also that $\Delta_X u_1 = \Delta_X(u_1+p \cdot X)$), we obtain
\begin{equation}\label{equ15}
    (\lambda+\delta ) \mes{P}=\chi  \int_P -\Delta_X(u_1+p \cdot X)dX +\int_P c(x_1,X,t)[p+D_X u_1])dX
\end{equation}
where $\mes{P}$ denotes the Lebesgue measure of $P$.
By using Green Formula 
\begin{equation*}
    -\chi  \int_P \Delta_X(u_1+p \cdot X)dX=-\chi \int_{\partial P }(D_X u_1+p).\tilde{N} d\sigma
\end{equation*}
where $\tilde{N}$ denotes the outward, unit normal to $\partial P$ and where
  $$\partial P = (\partial P \cap \partial\Omega)\cup (\partial P\cap \Omega)\; .$$
We first point out that, because of the periodicity of $u_1$ and the opposite orientation of the normal vector on both side of the cell
\begin{equation}\label{equ39}
  \chi \int_{(\partial P\cap \Omega)}(D_X u_1+p) \cdot \tilde{N} d\sigma=0.
\end{equation}
On the other hand, recalling the boundary condition of (\ref{cellpb})
\begin{equation*}
    -\chi \int_{\partial P \cap \partial\Omega }(D_X u_1+p) .\tilde{N} d\sigma= \int_{\partial P  \cap \partial\Omega} \, \Theta(x_1,X,t,\mu,\nu) d\sigma.
\end{equation*}

Next we consider the $c$-term : by integration by parts
  \begin{equation}
    \int_P c(x_1,X,t)D_X u_1dX =\int_{\partial P} u_1c(x_1,X,t).\tilde{N} d\sigma - \int_P u_1\div_X (c) dX. \nonumber
  \end{equation}
Because of {\bf (C2)}, the last integral of the right-hand side vanishes, while, for the first one, we use similar argument as above :  because of the periodicity properties of the velocity function $c$, the integral over $\partial P\cap \Omega$ is $0$ (the same reasons as for (\ref{equ39})) and by {\bf (C3)}, it is also the case for the integral over $\partial P \cap \partial\Omega$.

Gathering these informations, inserting them in (\ref{equ15}) and recalling the definition of $\Theta$, one gets
\begin{align}\label{equ16}
  (\lambda +\delta) \mes{P}=\int_{\partial P  \cap \partial\Omega} [\eta_p(x_1,X)\mu+&\eta_a(x_1,X,t)g_a(\mu)-\frac{\alpha}{\omega}\varrho(x_1,X)\nu  ] d\sigma \nonumber\\+&p \cdot \int_P c(x_1,X,t)dX .
\end{align}

In order to obtain the homogenized equation, we introduce
\begin{align}
  &\bar{c}(x_1,t)=\dfrac{1}{\mes{P}} \int_P c(x_1,X,t)dX, \nonumber \\
  &\bar{\eta}_p(x_1)=\dfrac{1}{\mes{\partial P \cap \partial \Omega }}\int_{\partial P \cap \partial \Omega }\eta_p(x_1,X)d\sigma \label{equ24}\\
  &\bar{\eta}_a(x_1,t)=\dfrac{1}{\mes{\partial P \cap \partial \Omega }}\int_{\partial P \cap \partial \Omega }\eta_a(x_1,X,t)d\sigma \nonumber\\
%  &\bar{\chi}=\chi \ratSV\\
  &\bar{\varrho}(x_1)=\dfrac{1}{\mes{\partial P \cap \partial \Omega }}\int_{\partial P \cap \partial \Omega }\varrho(x_1,X )d\sigma \nonumber
%&\bar{\varrho}_p(x_1)=\dfrac{1}{\mes{p}}\int_{p }\varrho(x_1,X)dX \nonumber,
\end{align}
where $\mes{\partial P \cap \partial \Omega}$ denotes the area of the surface $\partial P \cap \partial \Omega$. With the notation
$$
\bar{\Theta}(x_1,t,u,v):=
\bar{\eta}_p(x_1,X)u+\bar{\eta}_a(x_1,X,t)g_a(u)+\frac{\alpha}{\omega}\bar{\varrho}(x_1,X)v\; ,
$$
$$
\ratSV : = \frac{\mes{\partial P \cap \partial \Omega}}{\mes{P}}\; ,
$$
we get
\begin{equation}
 \lambda= \ratSV\bar{\Theta}(x_1,t,\mu,\nu)
+\bar{c}(x_1,t) \cdot p-\delta \label{equ40},
\end{equation}
The one dimensional averaged equation of transport and absorption of nutrients is thus obtained by inserting the value of $\lambda$ and $p$ in the equation (\ref{equ40})
\begin{equation}\label{nuthom}
  \dfrac{\partial u}{\partial t}+\bar{c}(x_1,t)\cdot e_1 \dfrac{\partial u}{\partial x_1}=
\zeta(x_1,t)\varphi(v)-\ratSV \bar{\Theta}(x_1,t,u,v)
\end{equation}
The term $\displaystyle \ratSV \bar{\Theta}(x_1,t,u,v)$ represents the global result of the different phenomena on the boundary of the small intestine : production of nutrients by surfacic degradation, active and passive absorption. The interesting feature in this term comes from the coefficient $\displaystyle \ratSV$ which measures the ratio between the large surface of the villi compared to the relatively small volume of each cell. It therefore describes the effect of the geometry of the villi on the absorption and degradation processes. 

The term $\displaystyle \ratSV[\bar{\eta}_p(x_1)u+\bar{\eta}_a(x_1,t)g(u)]$ gives an averaged value of absorption by intestinal wall, which takes into account the effect of villi folds  as well as the differences between passive and active absorption. 

In the same way as for the nutrients $u^\e$, we may obtain the one dimensional homogenized equation for feedstuffs $v^\e$
  \begin{equation}\label{feedhom}
  \dfrac{\partial v}{\partial t}+\bar{c}(x_1,t)\cdot e_1 \dfrac{\partial v}{\partial x_1}=
-\zeta(x_1,t)\varphi(v)-\ratSV\frac{1}{\omega}\bar{\varrho}(x_1)v.
  \end{equation}

In order to compare the homogenized equations (\ref{nuthom})-(\ref{feedhom}) with the models presented in \cite{TLGLB}, we recall  that, roughly speaking, in these models, the bolus is identified as a  cylinder of fixed length and variable radius $r$, composed of a single  feedstuff $A$ which is transformed into an absorbable nutrient $B$  through different types of enzymatic degradations. In fact, the main  model is more sophisticated since $A$ and $B$ can appear under several  forms (typically for $A$ a solubilized and a non-solubilized form).

Two degradation mechanisms are taken into account :  a ``volumic'' one  taking place inside the bolus and resulting from the action of  pancreatic and gastric enzymes and a ``surfacic'' one taking place on  the villi and resulting from the action of the brush-border enzymes.  Then, once the absorbable nutrient $B$ reaches the surface of bolus,  hence the intestinal wall, the absorption is ensured by a Michaelis- Menten mechanism. Therefore, even if the above $3$-d model is very  simplified, the functions $v$ and $A$ have the same nature and  represent the large particles of feedstuffs, as well as the functions  $u$ and $B$ represent the absorbable nutrients. Furthermore the $3$-d  model described the same phenomena, at least on the boundary.

Therefore, as we already mentioned it in the introduction, the above  homogenization process  equations (\ref{nuthom})-(\ref{feedhom})  justifies the rather simple form of the equations presented  in  \cite{TLGLB} : as long as we are just interested in ``macroscopic''  phenomena, it is reasonnable to describe the effects of the complex  geometry of the villi, the different types of degradation and the  absorption process by these odes
\begin{remark} In the above analysis, the effects of villi is summarized and measured by the (a priori large) $\displaystyle \ratSV$-coefficient which described the consequences of their particular finger-like geometry. It is worth pointing out that this type of analysis can be used as well to understand the effects of villi in the intestinal tract but also the effects of micro-villi inside the villi.
\end{remark}

% The homogenized equations of digestion, (\ref{feedhom}) and (\ref{nuthom}), can be compared with the model (1) of digestion  which is the first and simplest model among the 4 models proposed in \cite{TLGLB}.
% In this model, the bolus is a cylinder of fixed length and variable radius $r$, composed of $A$ and $B$. The result of surfacic and volumic degradation is the transformation of $A$ to nutrients $B$. Once $B$ is on the surface of bolus, the Michaelis-Menten mechanism ensures the absorption through intestinal wall. Therefore, the function $v$ and $A$ have the same nature and represent the large molecules of feedstuffs, this is also the case of $u$ and $B$ which represents the absorbable nutrients. The rate of volumic and surfacic degradation in Model (1) of \cite{TLGLB} is an example of the general functions which represent the different rates of degradation in the volume and on the surface of the bolus. We can describe the In the same way, the absorption rate can be described. It means that the Michaelis-Menten mechanism of absorption in Model (1) is an example of the rigorous rate of absorption obtained in this section.
\subsection{The Rigorous Result and Proof}\label{rr}
We are now in position to state the rigorous result.
\begin{theorem}\label{conv}Assume that $\Omega$ is a $C^2$-domain satisfying the properties described in Section~\ref{3.2}, that {\bf (C1)}-{\bf (C3)}, {\bf (T1)-(T2)} holds and that $u_0, v_0$ are continuous functions such that $u_0(0)=v_0(0) =0$.
Then the sequences $(\ue,\ve)_{\epsilon}$ converge locally uniformly, as $\epsilon \rightarrow 0$, to the unique (viscosity) solution $(u,v)$ of the system
\begin{equation}\label{equ18}
 \left\{
\begin{aligned}
  \dfrac{\partial u}{\partial t}+\bar{c}(x_1,t)\cdot e_1 \dfrac{\partial u}{\partial x_1}=&
\zeta(x_1,t)\varphi(v)-\ratSV \bar{\Theta}(x_1,t,u,v) \; 
\text{  in}\,\, Q_T \\
 \dfrac{\partial v}{\partial t}+\bar{c}(x_1,t)\cdot e_1 \dfrac{\partial v}{\partial x_1}=&
-\zeta(x_1,t)\varphi(v)-\ratSV\frac{1}{\omega}\bar{\varrho}(x_1)v
 \; 
\text{  in}\,\, Q_T \\
u(0,t)=u_0(t) \;\hbox{ and }&\; v(0,t)=v_0(t)\quad \quad \text{on } \,\,\partial Q_T\\
u(x_1,0)=v(x_1,0)=& \; 0 \quad \quad \text{in  } [0,+\infty)
\end{aligned}\right.
\end{equation}
\noindent where $Q_T=(0,+\infty) \times (0,T)$ and $\partial Q_T =\{x_1 = 0, t \in (0,T)\}$.
\end{theorem}

The averaged problem (\ref{equ18}) can be seen as a simplified version of the more complicated initial-boundary value problem (\ref{equfeed})-(\ref{Pylnut}) : it is clearly easier to compute the solution of (\ref{equ18}) than to take into account the complex geometry and boundary condition of (\ref{equfeed})-(\ref{Pylnut}).

\begin{proof}[Proof of Theorem~\ref{conv}]
Before providing the proof, we make some remarks about the existence and uniqueness of $\ue$ and $\ve$. The system \ref{equfeed})-(\ref{Pylnut}) is in fact decoupled and therefore we prove (by similar methods) the existence and uniqueness of $\ve$ and then of $\ue$.

The initial-boundary value problem (\ref{equfeed})-(\ref{equfeedb})-(\ref{emptySI})-(\ref{Pylfeed}) is a classical parabolic problem with Dirichlet and Neumann boundary conditions : it therefore admits smooth solutions. If one does not insist on proving the existence of smooth solutions, the existence and uniqueness of a viscosity solution of this problem can also be obtained by easier viscosity solutions arguments, using Perron's method (cf. \cite{Ishii(1987)}, \cite{CrandallN.S.}) and comparison results (\cite{Barles1999},  \cite{CrandallN.S.}). Of course, the result for (\ref{feedtonut1})-(\ref{equnutb})-(\ref{emptySI})-(\ref{Pylnut}) follows from the same arguments.

Applying the Maximum Principle (or a comparison result for viscosity solutions), it is easy to prove that $0\leq \ve (x,t) \leq ||v_0 ||_\infty $ in $\overline{\Omega}_\e \times [0,T]$ since $0$ and $||v_0 ||_\infty$ are respectively subsolution and supersolution of (\ref{feedtonut1})-(\ref{equnutb})-(\ref{emptySI})-(\ref{Pylnut}) . In particular, the $\ve$ 's are uniformly bounded. For the $\ue$, the situation is unfortunately a little bit more complicated : since $0$ is a subsolution of (\ref{feedtonut1})-(\ref{equnutb})-(\ref{emptySI})-(\ref{Pylnut}), we have $ \ue (x,t) \geq 0 $ on $\overline{\Omega}_\e \times [0,T]$ but it is not obvious at all to get an upper bound. For the time being, we assume that the $\ue$ 's are uniformly bounded and we will come back on this point at the end of the proof.

We provide the full convergence proof only in the case of the $u^\epsilon$'s, the one for the $v^\epsilon$ being obtained by similar and even simpler argument. In this proof, because of the decoupling of our system, we assume that we already know that the $v^\epsilon$'s are converging uniformly.

In order to prove the convergence of $u^\epsilon$ towards $u$, we use the standard method in such problems : we combine the half-relaxed limit method \cite{Barles1995,CrandallN.S.} with the Perturbed Test-Function method introduced by L. C.~Evans \cite{Evans1989}.  It is worth pointing anyway that the non-classical feature in our result and proof comes from the $3$-d to $1$-d passage to the limit and the change in the nature of the problem.

To this end, we introduce
$$\bar{u}(x,t)=\limsup_{\substack{\epsilon \to 0,y\to x\\ s\to t}}u^\epsilon(y,s)\quad, \quad\underline{u}(x,t)=\liminf_{\substack{\epsilon \to 0,y\to x \\ s\to t}}u^\epsilon(y,s) \; .$$
We have to prove that $\bar{u}$ is a subsolution of ($\ref{equ18})$ and $\underline{u}$ is a supersolution of $(\ref{equ18})$; since the proofs for the sub and supersolution cases are similar, we only present the arguments  for the subsolution case.

Let $\phi: [0,+\infty)\times [0,T] \rightarrow \R $ be a smooth test-function and $(x_1^0,t_0)$ be a strict maximum point of $\bar{u}-\phi$. In order to prove that $\bar{u}$ is a subsolution of $(\ref{equ18})$, we first consider the case when $x_1^0 >0,t_0>0$ where we have to prove
\begin{align}\label{equ20}
\dfrac{\partial\phi }{\partial t}(x_1^0,t_0)+& \bar{c}(x_1^0,t_0)\cdot e_1 \dfrac{\partial \phi }{\partial x_1}(x_1^0,t_0)\leq \nonumber \\& \zeta(x_1^0,t_0)\phi(v) -\ratSV\bar{\Theta}(x_1^0,t_0,u(x_1^0,t_0),v(x_1^0,t_0))
\end{align}

To apply the perturbed test-function method, we need the
%First of all, we prove that, for each $\mu$ and $p$, there exist a unique $\lambda$ such that the ergodic problem (\ref{cellpb}) has a solution.
\begin{lemma}\label{prop} The cell problem (\ref{cellpb}) has a $X_1$-periodic solution $u_1$ if and only if the parameters $\lambda, p , \mu, \nu, \delta, x_1, t$ satisfy Equation~(\ref{equ40}). Moreover this solution is unique up to an additive constant.
\end{lemma}
\begin{proof} The proof is standard and relies on the Fredholm alternative. By {\bf (C2)}, the operator $\mathcal{L}:= - \chi \Delta - c\cdot D$ together with Neumann boundary conditions, is self-adjoint and the Strong Maximum Principle shows that the kernel of this operator only contains the constant functions. On the other hand, the computations of the previous subsection ensure that the right hand side of equation (\ref{cellpb}) (including the boundary condition) is orthogonal to the constant functions, i.e. the kernel of $\mathcal{L}$, if and only if Equation~(\ref{equ40}) holds. Therefore this condition implies the existence of a solution of (\ref{cellpb}), which is $C^2$ by using standard elliptic regularity. This solution is of course unique up to an additive constant because of the structure of the kernel.  
\end{proof}

We pick some constant $0<\gamma \ll 1$. In view of Lemma~\ref{prop}, for the choice of the parameters $x_1^0$, $t_0$, $\delta :=\zeta(x_1^0,t_0)\varphi(v(x_1^0,t_0))$
\begin{equation}\label{equ23}
\nu= v(x_1^0,t_0) \; , \;  \mu = u(x_1^0,t_0)-\gamma \; , \; p= \dfrac{\partial \phi}{\partial x_1}(x_1^0,t_0)e_1 
\end{equation}
and if we choose $\lambda$ given by Equation~(\ref{equ40}), there exists a smooth solution $u_1 (X)$ of (\ref{cellpb}) associated to these parameters. 

We use this function to introduce the perturbed test-function $\phi ^\epsilon$ 
$$\phi ^\epsilon (x,t)=\phi (x_1,t)+\epsilon u_1 (\dfrac{x}{\epsilon}) \; .$$ 

By standard results (\cite{Barles1995}, p.88), for $\e$ small enough, there exists a maximum point $(x^\epsilon,t^\epsilon)$ of $u^\epsilon -\phi^\epsilon$ near $((x_1^0,0,0), t_0) $. Moreover
\begin{equation}\label{equ22}
x^\epsilon \to (x_1^0,0,0) ,\quad  t^\epsilon \to t_0\quad\text{as} \quad \epsilon \to 0
\end{equation}
\begin{align}\label{equ21}
u^\epsilon(x^\epsilon, t^\epsilon) \to \bar{u}(x_1^0,t_0)\quad \text{as} \quad \epsilon \to 0
\end{align}

First, we prove that the maximum point $(x^\epsilon,t^\epsilon)$ can not be on the boundary for $\e$ small enough. Otherwise, if $(x^\epsilon,t^\epsilon) \in \partial \Omega_\epsilon \times (0,T)$, then, by the maximum point property on the boundary
$$\dfrac{\partial }{\partial n} (u^\epsilon(x^\epsilon,t^\epsilon)-\phi(x_1^\epsilon,t^\epsilon)
-\epsilon u_1(X^\epsilon))\geq 0 $$
where $X^\epsilon=x^\epsilon/\epsilon$, thus
$$
\dfrac{\partial u^\epsilon}{\partial n} (x^\epsilon,t^\epsilon)
-[\dfrac{\partial \phi}{\partial x_1}(x_1^\epsilon,t^\epsilon).e_1+D_X u_1(X^\epsilon)].n\geq 0.
$$
Using the smoothness of $\phi$ and recalling that $n (x^\epsilon) = N(X^\epsilon)$, we can write this inequality as
\begin{equation}\label{equ41}
\dfrac{\partial u^\epsilon}{\partial n} (x^\epsilon,t^\epsilon)
-[p+D_X u_1(X^\epsilon)]\cdot N\geq o(1),
\end{equation}
where, here and below, $o(1)$ denotes a quantity which goes to $0$ as $\e$ tends to $0$.

Recalling Equation (\ref{equnutb}) permits to obtain
\begin{equation}\label{equ19}
-\frac{1}{\chi}\Theta\left(x_1^\e,X^\e,t^\e,u^\epsilon(x_1^\e,t^\e),\ve(x_1^\e,t^\e)\right)-\bigg(p + D_X u_1(X^\epsilon)\bigg)\cdot N  \geq o(1).
\end{equation}
Because of (\ref{equ22})-(\ref{equ21}) and the continuity properties of the functions $\eta_p, \eta_a$, $\varrho$ and $g_a$ and because of the convergence of $\ve$, Equation (\ref{equ19}) gives
$$
-\Theta\left(x_1^0,X^\e,t_0,\bar{u}(x_1^0,t^0),v(x_1^0,t^0))\right)
-\chi\bigg(p + D_X u_1(X^\epsilon)\bigg)\cdot N  \geq o(1).\
$$%
Now we replace $\bar{u}(x_1^0,t^0)$ by $\mu+\gamma$
$$\Theta\left(x_1^0,X^\e,t_0,\mu+\gamma,v(x_1^0,t^0))\right)
-\chi\bigg(p + D_X u_1(X^\epsilon)\bigg)\cdot N  \geq o(1).$$
Because of the properties of $\eta_p, \eta_a$, $\varrho$ and $g_a$,  $\Theta(x_1,X,t,u,v)$ is a strictly increasing function in $u$ (uniformly wrt the other parameters); by using (\ref{cellpb}) together with (T2), we obtain
\begin{equation}
-\underline{\eta}\gamma \geq - \Theta\left(x_1^0,X^\e,t_0,\mu +\gamma ,v(x_1^0,t^0))\right)
+ \Theta\left(x_1^0,X^\e,t_0,\mu,v(x_1^0,t^0))\right)
 \geq o(1)
\end{equation}
which yields the contradiction.

Therefore the maximum point $(x^\epsilon, t^\epsilon)$ of $u_\epsilon-\phi_\epsilon$ is in $ \Omega_\epsilon \times (0,T)$. Since $u_\epsilon$ is a solution of (\ref{feedtonut1}), classical properties yield to the inequality 
\begin{align}
\dfrac{\partial \phi}{\partial t}(x^{\epsilon}_1,t^\epsilon)- \chi \Delta_{X} u_1(X^\epsilon)+&c(x^{\epsilon},X^\epsilon,t^\epsilon)\left(\dfrac{\partial \phi}{\partial x_1}(x^{\epsilon}_1,t^\epsilon)e_1
+D_X u_1(X^\epsilon)\right) \nonumber \\
& -\zeta(x_1^\e,t^\e)\varphi(v)\leq o(1).
\end{align}
For $\epsilon$ small enough, using (\ref{equ22}) and the regularity of $\phi $ imply
\begin{equation}
\dfrac{\partial \phi}{\partial t}(x^{0}_1,t^0)- \chi\Delta_{X} u_1(X^\epsilon)+c(x_1^{0},X^\epsilon,t_0)(p+D_Xu_1(X^\epsilon))-\delta \leq o(1).
\end{equation}
Furthermore, by Equation (\ref{cellpb}), 
$$\lambda= -\chi \Delta u_1+c(x_1^{0},X^\epsilon,t_0)(p+D_Xu_1)-\delta$$
which yields to the inequality
$$ \dfrac{\partial \phi}{\partial t}(x_1^0,t_0)+\lambda \leq o(1).$$
As we already mentioned it above, this inequality is equivalent to the inequality (\ref{equ20}) by just inserting the value of $\lambda$ from Lemma \ref{prop} into the above equation and letting $\e$ tend to $0$. And the proof of this first case is complete.

We should now consider the cases when the maximum point is achieved either for $t=0$ or at $x_1 = 0$ to complete the proof.

For the initial condition ($t=0$), a combination of the above proof and classical arguments shows that we have the viscosity inequality
\begin{equation*}
  \min\{\dfrac{\partial \bar{u}}{\partial t}+\bar{c}(0,t)\cdot e_1 \dfrac{\partial \bar{u}}{\partial x_1}-
\zeta(0,t)\varphi(v)+\ratSV\Theta (x_1,t,u,v),\bar{u}\}\leq0\, ,
  \end{equation*}
if $x_1>0$, while, for $x_1=0, t>0$, one has
\begin{equation*}
  \min\{\dfrac{\partial \bar{u}}{\partial t}+\bar{c}(x_1,t)\cdot e_1 \dfrac{\partial \bar{u}}{\partial x_1}-
\zeta(x_1,t)\varphi(v)+\ratSV
\Theta(x_1,t,u,v)),\bar{u}-u_0(t)\}\leq0\, ,
  \end{equation*}
and for the case $x_1=0, t=0$ --which is a priori a particular case--, since $u_0(0)=0$, we can still use one of these inequalities which are the same.

It is proved (cf. (\cite{Barles1995}p.99)\cite { CrandallN.S.}) that, if $\bar{u}$ is a subsolution of (\ref{equ18}), then the above initial condition in viscosity sense reduces in fact to a classical one 
$$\bar{u}(x_1^0,0)\leq 0 \quad \hbox{on  }[0,+\infty) .$$
Since $\bar{c}(x_1,t)\cdot e_1 >0$, the generalized Dirichlet condition reduces also to a classical one (cf. \cite{Barles1995}(cor 4.1 in p.169)), namely
$$\bar{u}(0,t)\leq u_0(t) \quad \hbox{on  }[0, T),$$
as a consequence of the fact that the characteristic are pointing outward the domain on the boundary.

To conclude the proof, we invoke a (strong) comparison result for (\ref{equ18}) : such result is classical and it yields $\bar{u} \leq \underline{u}$ on $[0,+\infty) \times [0,T]$, implying the desired convergence result.

It remains to prove that the $\ue$'s are indeed uniformly bounded. To this aim, we recall that $\Omega $ is a $C^2$-domain and therefore there exists a $x_1$-periodic, $C^2$-function $d :\bar{\Omega} \to [0,\infty)$ such that
$$Dd(x)\cdot N \leq -1 \quad \hbox{on}\quad \partial\Omega.$$
Because of the particular form of $\Omega$, the function $d$, as well as its first and second derivatives, are also bounded functions.

Now, we introduce the functions $w^\e\, :\, \bar{\Omega}_\e\times [0,T]\to [0,\infty)$ given by
 $$w^\e(x,t)=k_1+k_2t+k_3(\e||d||_\infty-\e d(x/\e)) \; ,$$
for some constants $k_1, k_2, k_3 \geq0.$
We first plug these functions into the boundary condition (\ref{equnutb}) : using that the $\ve$'s are bounded and that the absorption terms are positive, the supersolution condition is satisfied if we choose $k_3$ large enough. Then we consider Equation (\ref{feedtonut1}) : since $d$ has bounded first and second derivatives, the supersolution condition is also satisfied by choosing $k_2$ large enough. Finally we choose $k_1$ large enough to treat the boundary condition (\ref{Pylfeed}).

Applying the Maximum Principle (or a comparison result for viscosity solutions) gives $\ue (x,t) \leq w^\e (x,t) $ in $\overline{\Omega}_\e \times [0,T]$ and the proof is complete.
\end{proof}


\begin{thebibliography}{10}

\bibitem{Aris1956}
R.~Aris.
\newblock {On the Dispersion of a Solute in a Fluid Flowing through a Tube}.
\newblock {\em Royal Society of London Proceedings Series A}, 235:67--77, April
  1956.

\bibitem{Barles1995}
G.~Barles.
\newblock {\em Solutions de viscosit\'e des \'equations de Hamilton-Jacobi
  (Math\'ematiques et Applications) (French Edition)}.
\newblock Springer, 1 edition, March 1995.

\bibitem{Barles1999}
G.~Barles.
\newblock Nonlinear neumann boundary conditions for quasilinear degenerate
  elliptic equations and applications.
\newblock {\em Journal of Diff. Eqns}, 154:191--224, 1999.

\bibitem{BDLS}
G.~Barles, F.~Da~Lio, P.-L Lions, and P.~E. Souganidis.
\newblock {Ergodic problems and periodic homogenization for fully nonlinear
  equations in half-space type domains with Neumann boundary conditions}.
\newblock {\em ArXiv e-prints}, October 2009.

\bibitem{Bourgeat2003}
A.~Bourgeat, M.~Jurak, and A.~L. Piatnitski.
\newblock {Averaging a transport equation with small diffusion and oscillating
  velocity}.
\newblock {\em Mathematical Methods in the Applied Sciences}, 26:95--117,
  January 2003.

\bibitem{CrandallN.S.}
M.~G. Crandall, H~Ishii, and P.~L. Lions.
\newblock user's guide to viscosity solutions of second order partial
  differential equations.
\newblock {\em Bull. Amer. Math. Soc. (N.S.)}, Bull.:Amer.Math.Soc.27 1--67,
  1992.

\bibitem{Evans1989}
L.C. Evans.
\newblock The perturbed test function method for viscosity solutions of
  nonlinear {PDE}.
\newblock {\em Proc. Roy. Soc. Edinburgh Sect. A}, 111(3-4):359--375, 1989.

\bibitem{Evans1992}
L.C. Evans.
\newblock Periodic homogenisation of certain fully nonlinear partial
  differential equations.
\newblock {\em Proc. Roy. Soc. Edinburgh Sect. A}, 120(3-4):245--265, 1992.

\bibitem{Evans2010}
L.C. Evans.
\newblock {\em Partial Differential Equations: Second Edition (Graduate Studies
  in Mathematics)}.
\newblock American Mathematical Society, 2 edition, March 2010.

\bibitem{lect1}
G.~Allaire.
\newblock The theory of periodic homogenization.
\url{ http://www.cmap.polytechnique.fr/~allaire/homogenization.html}

\bibitem{Hoermander2003}
Lars H\"ormander.
\newblock {\em Lectures on Nonlinear Hyperbolic Differential Equations
  (Math\'ematiques et Applications)}.
\newblock Springer, 1 edition, December 2003.

\bibitem{Imbert2008}
C.~Imbert and R.~Monneau.
\newblock {Singular perturbations of nonlinear degenerate parabolic PDEs: a
  general convergence result}.
\newblock {\em Archive for Rational Mechanics and Analysis}, 187:49--89,
  January 2008.

\bibitem{Ishii(1987)}
H.~Ishii.
\newblock {\sl Perron's method for Hamilton-Jacobi Equations,}.
\newblock {\em Duke Math. J}, {55}:369--384., 1987.

\bibitem{Keener2008}
J.~Keener and J.~Sneyd.
\newblock {\em Mathematical Physiology (Interdisciplinary Applied Mathematics)
  2 Vol Set}.
\newblock Springer, 2nd edition, December 2008.

\bibitem{LOGAN2002}
J.~D. Logan, A.~Joern, and W.~Wolesensky.
\newblock Location, time, and temperature dependence of digestion in simple
  animal tracts.
\newblock {\em Journal of Theoretical Biology}, 216:5 -- 18, January 2002.

\bibitem{Mernone2002}
A.~V. Mernone, J.~N. Mazumdar, and S.~K. Lucas.
\newblock A mathematical study of peristaltic transport of a casson fluid.
\newblock {\em Mathematical and Computer Modelling}, 35(7-8):895 -- 912, 2002.

\bibitem{Miftahof2007}
R.~Miftahof and N.~Akhmadeev.
\newblock Dynamics of intestinal propulsion.
\newblock {\em Journal of Theoretical Biology}, 246(2):377 -- 393, 2007.

\bibitem{Piccinini1978}
Livio~C. Piccinini.
\newblock Homogeneization problems for ordinary differential equations.
\newblock {\em Rend. Circ. Mat. Palermo (2)}, 27(1):95--112, 1978.

\bibitem{Randall1997}
D.~Randall, W.~Burggren, K.~French, and R.~Eckert.
\newblock {\em Eckert Animal Physiology: Mechanisms and Adaptations}.
\newblock W.H. Freeman \& Company, 4th edition, 1997.

\bibitem{Rivest2000}
J.~Rivest, J.~F. Bernier, and C.~Pomar.
\newblock A dynamic model of protein digestion in the small intestine of pigs.
\newblock {\em J Anim Sci}, 78(2):328--340, February 2000.

\bibitem{Rudin1976}
W.~Rudin.
\newblock {\em Principles of Mathematical Analysis, Third Edition}.
\newblock McGraw-Hill Science/Engineering/Math, 3rd edition, January 1976.

\bibitem{Rudin1986}
W.~Rudin.
\newblock {\em Real and Complex Analysis (International Series in Pure and
  Applied Mathematics)}.
\newblock McGraw-Hill Science/Engineering/Math, 3 edition, 1986.

\bibitem{TLGLB}
M.~Taghipoor, P.~Lescoat, C.~Georgelin, JR. Licois, and G.~Barles.
\newblock {Mathematical Modeling of Transport and Degradation of Feedstuffs in
  the Small Intestine}.
\url{ http://hal.archives-ouvertes.fr/hal-00555287/en/}, 2011.

\bibitem{Tharakan}
Tharakan.
\newblock {\em Modelling of physical and chemical processes in the small
  intestine}.
\newblock PhD thesis, University of Birmingham, 2008.

\bibitem{YAMAUCHI(2007)}
K.e. Yamauchi.
\newblock Review of a histological intestinal approach to assessing the
  intestinal function in chickens and pigs.
\newblock {\em Animal Science Journal}, 78:356–370, 2007.

\bibitem{Zhao1997}
X.~T. Zhao, M.~A. McCamish, R.~H. Miller, L.~Wang, and H.~C. Lin.
\newblock Intestinal transit and absorption of soy protein in dogs depend on
  load and degree of protein hydrolysis.
\newblock {\em J Nutr}, 127(12):2350--2356, December 1997.

\end{thebibliography}
\end{document}